\documentclass[12pt]{amsart}

\usepackage{amssymb}
\usepackage{amsmath}
\usepackage{amsfonts}
\usepackage{graphicx}
\usepackage{amscd}
\usepackage{xcolor}
\usepackage{float}
\usepackage{mathtools}

\newtheorem{thm}{Theorem}[section]
\newtheorem{cor}[thm]{Corollary}
\newtheorem{lema}[thm]{Lemma}
\newtheorem{prop}[thm]{Proposition}
\theoremstyle{definition}
\newtheorem{defn}[thm]{Definition}

\theoremstyle{remark}

\def\R{\mathbb{R} }

\def\to{\rightarrow}

\def\phi{{\varphi} }

\begin{document}

\title{Problems involving the fractional $g$-Laplacian with Lack of Compactness}

\author{Sabri Bahrouni}
\address[S. Bahrouni]{LR18ES15, University of Monastir, Monastir, Tunisia}
\email{sabribahrouni@gmail.com}

\author{Hichem Ounaies}
\address[H. Ounaies]{Department of Mathematics, Faculty of Sciences,
University of Monastir, Monastir, Tunisia}
\email{hichem.ounaies@fsm.rnu.tn}

\author{Olfa Elfalah}
\address[O. Elfalah]{Department of Mathematics, Faculty of Sciences,
University of Monastir, Monastir, Tunisia}
\email{elfalaholfa328@gmail.com}

\keywords{Fractional Orlicz-Sobolev spaces, Compact embedding, Lions-type theorem, Variational Methods \\
\hspace*{.3cm} {\it 2010 Mathematics Subject Classifications}:
 46E30, 35R11, 45G05}

\begin{abstract}
In this paper we prove compact embedding of a subspace of the fractional Orlicz-Sobolev space $W^{s, G}\left(\mathbb{R}^{N}\right)$ consisting of radial functions, our target embedding spaces are of Orlicz type. Also, we prove a Lions and Lieb type results for $W^{s,G}\left(\mathbb{R}^{N}\right)$ that works together in a particular way to get a sequence whose the weak limit is non trivial. As an application,  we study the existence of solutions to Quasilinear elliptic problems in the whole space  $\mathbb{R}^N$ involving the fractional $g-$Laplacian operator, where the conjugated function $\widetilde{G}$ of $G$ doesn't satisfy the $\Delta_2$-condition.
\end{abstract}

\maketitle

\section{Introduction}

Recently, problems with lack of compactness received a great attention from researchers interested in the study of fractional and nonlocal elliptic. That's why the relevant literature looks daily increasing and numerous meaningful papers on this subject are by now available. For an
introduction to this class of problems and a comprehensive list of references we refer to the recent book \cite{Pucci-G}.\\

In this context, this paper is devoted to the study of the existence of solutions of a Quasilinear elliptic problems in the whole space  $\mathbb{R}^N,$ $N$ an integer $N\geq 3$, which will cause lack of compactness. We shall consider the following nonlinear fractional $g-$Laplacian elliptic problem
\begin{align} \label{m.equation}
\begin{cases}
(-\Delta_g)^s u +g(u)\frac{u}{|u|}=f(u)&\text{ in } \mathbb{R}^N\\
u\in W^{s,G}(\mathbb{R}^N).
\end{cases}
\end{align}

Here, the fractional $g-$Laplacian is defined as
$$
(-\Delta_g)^s u:= \,2 \text{ p.v.} \int_{\mathbb{R}^N} g\left( |D_s u|\right)\frac{D_s u}{|D_s u|} \frac{dy}{|x-y|^{N+s}},
$$
where $G$ is a Young function (see Section \ref{sec.prel} for the precise definition) such that $g=G'$ and $D_s u:=\frac{u(x)-u(y)}{|x-y|^s}$.  Clearly, when $G$ is a power function, $(-\Delta_{g})^s$ boils down to the fractional $p-$Laplacian  and to the $p-$Laplacian when $s\uparrow 1$. See \cite{FBS}.

Throughout this article we assume the Young function $G(t)=\int_0^tg(s)\,ds$ satisfies the following structural conditions:
\begin{equation} \label{G1} \tag{$G_1$}
1< p^-\leq \frac{tg(t)}{G(t)} \leq p^+<\infty ,\quad \text{for all }t>0.
\end{equation}
\begin{equation} \label{G2} \tag{$G_2$}
t\mapsto G(\sqrt{t}),\ t\in[0,\infty)  \text{ is convex}.
\end{equation}

The nonlinearity $f\colon\mathbb{R}\rightarrow\mathbb{R}$ assumed to be continuous
\begin{equation}\label{f1}\tag{$f_1$}
  \lim_{t\to 0}\frac{f(t)}{g(t)}=0,
\end{equation}
\begin{equation}\label{f2}\tag{$f_2$}
\limsup_{t\to +\infty}\frac{|f(t)|}{m(|t|)}<+\infty,
\end{equation}
where $m\colon (0,+\infty)\to \mathbb{R}$ is continuous function satisfying:
\begin{equation}\label{m1}\tag{$m_1$}
0<m^-\leq\frac{tm(t)}{M(t)}\leq m^+,\quad\text{for all } t>0,
\end{equation}
 where $p^+<m^-<m^+<p^-_*:=\frac{Np^-}{N-sp^-}$ and $M(t)=\int^t_0m(s)\ ds$ is a Young function.

There is $\theta>p^+,$ such that
\begin{equation}\label{f3}\tag{$f_3$}
0<\theta F(t)\leq f(t)t,\quad\text{for all }t\in\mathbb{R}\backslash\{0\},
\end{equation}
where
\begin{equation}\label{primitive}
    F(t)=\int_{0}^{t}f(\tau)\,d\tau.
\end{equation}

The unboundedness of the domain generally prevents the study of these types of problems by general methods of nonlinear analysis due to the lack of compactness.
In many different cases, it has been observed that by restricting to sub-spaces formed by functions respecting some symmetries of the problem, some forms of compactness were obtained.
Denote by
$$
W^{s,G}_{rad}(\mathbb{R}^{N})=\{u\in W^{s,G}(\mathbb{R}^{N}):\ u\ \text{is radially symmetric}\}.
$$
By $u$ is radially symmetric we mean, a function $u:\mathbb{R}^N\rightarrow\mathbb{R}$ satisfying $u(x)=u(y)$ for all $x,y\in\mathbb{R}^N$ such that $|x|=|y|$.
Strauss\cite{Strauss} was the first who observed that there exists an interplay between the regularity of the function and its radially symmetric property. Later on, Strauss's result was generalized in many directions, see \cite{Alves1, giovanni1, Sabri6, Fan1, Lions, daniel, napoli, Morry} for a survey of related results and elementary proofs of some of them within the framework of different spaces.\\

To the author's best knowledge, there have no works dealing with Strauss's results for the fractional Orlicz-Sobolev spaces.
Thus, the first step to deal with problem \eqref{m.equation} is clearly highlight the compactness result of Strauss. Our target embedding spaces are of Orlicz type.

\begin{thm}\label{Strauss-embedding}[{\it {\bf Strauss radial embedding}}]
Assume that \eqref{G1} hold and let $\Psi$ be a Young function verifying the $\triangle_{2}$-condition with
\begin{equation}\label{SE1}
  \lim_{|t|\rightarrow0}\frac{\Psi(t)}{G(t)}=0,
\end{equation}
and
\begin{equation}\label{SE2}
  \Psi\prec\prec G_*.
\end{equation}
Then, the embedding $$W^{s,G}_{rad}(\mathbb{R}^{N})\hookrightarrow L^{\Psi}(\mathbb{R}^{N})$$ is compact.
\end{thm}

A powerful tool in the proof of the previous Theorem is Lion's type result which is proved recently in \cite{Sabri7}.

\begin{thm}\label{Lions1}
Suppose that \eqref{G1} hold and let $\Psi$ be a Young function  satisfying \eqref{SE1} and \eqref{SE2}. Let $(u_{n})$ be a bounded sequence in $W^{s,G}(\mathbb{R}^{N})$ in such way that $u_n \rightharpoonup 0$ in $W^{s,G}(\mathbb{R}^{N})$ and
$$
\lim_{n\rightarrow+\infty}\left[\sup_{y\in\mathbb{R}^{N}}\int_{B_{r}(y)}G(u_{n})\,\mathrm{d}x \right]=0,
$$
for some $r>0$. Then, $u_{n}\rightarrow0$ in $L^{\Psi}(\mathbb{R}^{N})$.
\end{thm}

Problem \eqref{m.equation} has a clear variational structure, indeed its weak solutions (witch will be defined later in section 2) are critical points of the following energy functional defined on $ W^{s,G}(\mathbb{R}^{N})$ by

\begin{equation}\label{functional}
  \mathcal{T}(u)=\mathcal{J}(u)+\mathcal{I}(u)-\mathcal{F}(u)
\end{equation}
where
\begin{equation}\label{func.1}
 \mathcal{F}(u):=\int_{\R^N}F(u)\,dx,\quad F\ \text{ is given in}\ \eqref{primitive}
\end{equation}
and
\begin{equation} \label{func.2}
\mathcal{J}(u):=\iint_{\R^N\times\R^N} G(|D_s u|)\frac{dxdy}{|x-y|^N},\qquad \mathcal{I}(u):=\int_{\R^N} G(|u|)\,dx.
\end{equation}

Our first main result for Problem \eqref{m.equation} is the following

\begin{thm}\label{m.r.1}
Let hypotheses  \eqref{G1}-\eqref{G2} and \eqref{f1}-\eqref{f3} hold. Then, problem \eqref{m.equation} has a nontrivial weak solution.
\end{thm}

We prove the existence result stated in Theorem \ref{m.r.1}, by applying the Mountain Pass Theorem (see \cite{Amb-Rab}) to the energy functional $\mathcal{T}$ defined in \eqref{functional}, we get a critical point $u\in W^{s,G}_{rad}(\R^N)$. Than using the Principle of Symmetric Criticality by Palai (see \cite{Palai}) we see that $ W^{s,G}_{rad}(\R^N)$  is a natural constraint for the functional $\mathcal{T},$ i.e. critical points of $\mathcal{T}$ constrained on $ W^{s,G}_{rad}(\R^N)$  are actually critical points of $\mathcal{T}$ on $ W^{s,G}(\R^N).$

\vskip1cm
In the second part of the paper, we intend to work with a situation that:

\begin{itemize}
  \item  $\tilde{G}$, the complementary function of $G$, does not satisfy the $\Delta_2$-condition which implies that the space $W^{s,G}\left(\mathbb{R}^{N}\right)$ is not reflexive any more (which is the case $p^-=1$);
  \item the function $t\mapsto G(\sqrt{t})$ does not necessarily to be convex.
\end{itemize}

But in this way, we will encounter some new challenges when dealing with our problem, and we need to extend the assumption \eqref{G1} to the following one
\begin{equation} \label{G1p} \tag{$G_1^{\prime}$}
1\leq p^-\leq \frac{tg(t)}{G(t)} \leq p^+<\infty, \quad \text{for all } t>0.
\end{equation}
Also, a new version  of Lions type result for fractional Orlicz-Sobolev space $W^{s,G}\left(\mathbb{R}^{N}\right)$ need to be established, whose the Young function $G$ does not need to satisfy the $\Delta_2$-condition.

\begin{thm}[{\it {\bf A Lions type result}}]\label{Lions2}
Let $G$ and $\Psi$ be Young functions such that
$$\lim\limits_{\substack{ t\to 0}}\frac{\Psi\left(\vert t\vert\right)}{G\left(\vert t\vert\right)}=\lim\limits_{\substack{ t\to \infty}}\frac{\Psi\left(\vert t\vert\right)}{G_{*}\left(\vert t\vert\right)}=0,$$
where $G_{*}$ is the Sobolev conjugate function of $G$ $($see section $2)$. Let $(u_n) \subset W^{s,G}(\mathbb{R}^{N})$ is a sequence such that
$$
\left(\int_{\mathbb{R}^{N}} G\left(\vert u_n\vert\right)dx\right)_n\quad\text{and}\quad \left(\int_{\mathbb{R}^{N}} G_{*}\left(\vert u_n\vert\right)dx\right)_n
$$ are bounded, and for each $\tau > 0$ we have
\begin{equation} \label{lim}
\operatorname{mes}\left(\left[\left|u_{n}\right|>\tau\right]\right)\rightarrow 0, \quad \text{as } n \rightarrow +\infty,
\end{equation}
then
$$ \left(\int_{\mathbb{R}^{N}} \Psi\left(\vert u_n\vert\right)dx\right)\rightarrow 0, \quad \text{as } n \rightarrow +\infty.$$
\end{thm}

  Our next result is a Lieb type result for fractional Orlicz-Sobolev space $W^{s, G}\left(\mathbb{R}^{N}\right)$ that works together with the Lions type result, in a particular way to get a sequence whose the weak
limit is non trivial. We want to remark that in \cite{Alves-Marcos} the author proved
a version for Orlicz-Sobolev spaces, but his approach does not include fractional
versions of these spaces.

\begin{thm}[{\it {\bf A Lieb type result}}]\label{Lieb}
   Let $G \in C^{1}\left([0,+\infty)\right)$ be a Young function satisfying \eqref{G1p} and $\left(u_{n}\right) \subset$ $W^{s, G}\left(\mathbb{R}^{N}\right)$ such that $\Phi_{s,G,\mathbb{R}^{N}}(u_n)^{1/p^{\pm}} \leq M$ $($where $\Phi_{s,G,\mathbb{R}^{N}}$ will be defined in section $2)$. Let $\tau,\ \delta>0$ such that
$$
\operatorname{mes}\left( \left[\left|u_{n}\right|>\tau\right]\right) \geq \delta, \quad \text{for all }n \in \mathbb{N},
$$
then there is $\left(y_{n}\right) \subset \mathbb{Z}^{N}$ such that $v_{n}(x)=u_{n}\left(x+y_{n}\right)$ has a subsequence whose its limit in $L_{l o c}^{G}\left(\mathbb{R}^{N}\right)$ is non trivial.
\end{thm}

Our second main result for Problem \eqref{m.equation} is the following

\begin{thm}\label{m.r.2}
  Let $G \in C^{1}\left([0,+\infty)\right)$ be a Young function  satisfying \eqref{G1p} and $f$ satisfies \eqref{f1}-\eqref{f3} hold.  Then, problem \eqref{m.equation} has a nontrivial weak solution.
\end{thm}

\section{Preliminaries}\label{sec.prel}
In this section, we will define the fractional order Orlicz-Sobolev spaces and we introduce some technical results that will be used throughout the paper.
\subsection{Young functions}

An application $G\colon\R_{+}\to \R_{+}$ is said to be a  \emph{Young function} if it admits the integral formulation $G(t)=\int_0^t g(\tau)\,d\tau$, where the right continuous function $g$ defined on $[0,\infty)$ has the following properties:
\begin{align*}
&g(0)=0, \quad g(t)>0 \text{ for } t>0 \label{g0} \tag{$g_1$}, \\
&g \text{ is non-decreasing on } (0,\infty) \label{g2} \tag{$g_2$}, \\
&\lim_{t\to\infty}g(t)=\infty  \label{g3} \tag{$g_3$} .
\end{align*}
From these properties, it is easy to see that a Young function $G$ is continuous, nonnegative, strictly increasing and convex on $[0,\infty)$.

The following  properties on Young functions are well-known. See for instance \cite{FBS} for the proofs.

\begin{lema} \label{lema.prop}
Let $G$ be a Young function satisfying \eqref{G1} and $a,b\geq 0$. Then
\begin{align*}
  &\min\{ a^{p^-}, a^{p^+}\} G(b) \leq G(ab)\leq   \max\{a^{p^-},a^{p^+}\} G(b),\tag{$L_1$}\label{L1}\\
  &G(a+b)\leq  2^{p^+} (G(a)+G(b)),\tag{$L_2$}\label{L2}\\
	&G \text{ is Lipschitz continuous}. \tag{$L_3$}\label{L_3}
 \end{align*}
\end{lema}




The \emph{complementary Young function} $\tilde G$ of a Young function $G$ is defined as
$$
\tilde G(t):=\sup\{tw -G(w): w>0\}.
$$
The functions $G$ and $\tilde{G}$
are complementary each other and satisfy the inequality below
\begin{equation} \label{ineb}
\tilde{G}\left(g(t)\right)\leq G(2t),\qquad \text{for all } t > 0.
\end{equation}
Moreover, we also have a Young type inequality given by
\begin{equation} \label{Young}
ab\leq G(a)+\tilde G(b)\qquad \text{for all }a,b\geq 0,
\end{equation}
and the following H\"older's type inequality
$$
\int_\Omega |uv|\,dx \leq \|u\|_{L^G(\Omega)} \|v\|_{L^{\tilde G}(\Omega)}
$$
for all $u\in L^G(\Omega)$ and $v\in L^{\tilde G}(\Omega)$. Moreover, it is not hard to see that $\tilde G$ can be written as
\begin{equation} \label{xxxx}
\tilde G(t)=\int_0^t \tilde{g}(\tau)\,d\tau,
\end{equation}
where $\tilde{g}(t)=\sup\{s:\ g(s)\leq t\}$. If $g$ is continuous then $\tilde{g}$ is the inverse of $g$.

\subsection{Fractional Orlicz-Sobolev spaces}

Given a Young function $G$ such that $G'=g$, a parameter $s\in(0,1)$ and an open set $\Omega\subseteq \R^N$. We consider the spaces
\begin{align*}
&L^G(\Omega) :=\left\{ u\colon \Omega \to \R \text{ measurable  such that }  \Phi_{G,\Omega}\left(\lambda u\right) < \infty \text{ for some }\lambda>0 \right\},\\
&W^{s,G}(\Omega):=\left\{ u\in L^G(\Omega) \text{ such that } \Phi_{s,G,\Omega}\left(\lambda u\right) < \infty \text{ for some }\lambda>0 \right\},
\end{align*}

where the modulars $\Phi_{G,\Omega}$ and $\Phi_{s,G,\R^N}$ are defined as
\begin{align*}
&\Phi_{G,\Omega}(u):=\int_{\Omega} G(|u(x)|)\,dx\\
&\Phi_{s,G,\Omega}(u):=
  \iint_{\Omega\times\Omega} G( |D_su(x,y)|)  \,d\mu,
\end{align*}
where  the \emph{$s-$H\"older quotient} is defined as
$$
D_s u(x,y):=\frac{u(x)-u(y)}{|x-y|^s},
$$
being $d\mu(x,y):=\frac{ dx\,dy}{|x-y|^N}$.
These spaces are endowed with the so-called \emph{Luxemburg norms}
\begin{align*}
&\|u\|_{L^G(\Omega)} := \inf\left\{\lambda>0\colon \Phi_{G,\Omega}\left(\frac{u}{\lambda}\right)\le 1\right\},\\
&\|u\|_{W^{s,G}(\Omega)} := \|u\|_{L^G(\Omega)} + [u]_{W^{s,G}(\Omega)},
\end{align*}
where the  {\em $(s,G)$-Gagliardo semi-norm} is defined as
\begin{align*}
&[u]_{W^{s,G}(\Omega)} :=\inf\left\{\lambda>0\colon \Phi_{s,G,\Omega}\left(\frac{u}{\lambda}\right)\leq 1\right\}.
\end{align*}
Under the assumption \eqref{G1}, the space $W^{s,G}(\Omega)$ is a reflexive Banach space. Moreover $C_c^\infty(\R^N)$ is dense in $W^{s,G}(\R^N)$. Also, the spaces $L^G(\Omega)$ and $W^{s,G}(\Omega)$ coincide with the set of measurable functions $u$ on $\Omega$ such that $\Phi_{G,\Omega}\left( u\right) < \infty$ and the set of functions in $L^G(\Omega)$  such that $ \Phi_{s,G,\Omega}\left( u\right) < \infty$ respectively. See \cite[Proposition 2.11]{FBS} and \cite[Theorem 4.7.3]{FS} for details.


The space of fractional Orlicz-Sobolev functions is the appropriated one to define the \emph{fractional $g-$Laplacian operator}
$$
(-\Delta_g)^s u :=2 \,\text{p.v.} \int_{\R^N} g( |D_s u|) \frac{D_s u}{|D_s u|} \frac{dy}{|x-y|^{N+s}},
$$
where \text{p.v.} stands for {\em in principal value}. This operator  is well defined between $W^{s,G}(\R^N)$ and its dual space $W^{-s,\tilde G}(\R^N)$ (see \cite{FBS} for details). In fact,
$$
\langle (-\Delta_g)^s u,v \rangle =   \iint_{\R^N\times\R^N} g(|D_s u|) \frac{D_s u}{|D_s u|}  D_s v \,d\mu,\quad\text{for all }v\in W^{s,G}(\R^N).
$$

In order to state some embedding results for fractional Orlicz-Sobolev spaces we recall that given two Young functions $A$ and $B$, we say that \emph{$B$ is essentially stronger than $A$} or equivalently that \emph{$A$ decreases essentially more rapidly than $B$}, and denoted by $A\prec \prec B$, if for each $a>0$ there exists $x_a\geq 0$ such that $A(x)\leq B(ax)$ for $x\geq x_a$. This is the case if and only if for
every positive constant $\lambda$,
\begin{equation} \label{do2}
 \lim\limits_{\substack{  t\to \infty}} \frac{A(\lambda t)}{B( t)}=0.
 \end{equation}
Suppose that
\begin{equation}
\int_0^1 \frac{G^{-1}(\tau)}{\tau^{(N+s)/N}} d\tau <\infty\quad\text{and}\quad
\int_1^{+\infty} \frac{G^{-1}(\tau)}{\tau^{(N+s)/N}} d\tau =\infty,
\end{equation}
we consider the function $G_*$, the Sobolev conjugate
function of $G$, defined by
\begin{equation}
G_*^{-1}(t)=\int_0^t \frac{G^{-1}(\tau)}{\tau^{(N+s)/N}} d\tau,\qquad \text{for }t>0.
\end{equation}

With these preliminaries the following compact and continuous embeddings hold.

\begin{prop}\label{ceb}$($see \cite{Sabri2}$)$
	Let $G$ be a Young function satisfying \eqref{G1} and let $s\in(0,1)$.
	\begin{itemize}
  \item[(i)]If $\Omega$
 is an open bounded set in $\mathbb{R}^{N}$, then the embedding
 $$W^{s,G}(\Omega)\hookrightarrow L^{G}(\Omega)$$ is compact. 
\item[(ii)] If $B$ is a Young function satisfies the $\Delta_2$-condition such that $B \prec \prec G_{*}$, then the embedding $$W^{s,G}(\R^N)\hookrightarrow L^{B}(\R^N)$$ is continuous. In particular, we have the continuous embeddings
$$W^{s,G}(\R^N)\hookrightarrow L^{G}(\R^N)$$
and $$W^{s,G}(\R^N)\hookrightarrow L^{G_{*}}(\R^N).$$
\end{itemize}
\end{prop}

A relation between modulars and norms holds. See \cite[Lemma 2.1]{Fukagai} and \cite[Lemma 3.1]{Sabri1}.
\begin{lema}\label{ineq1}
   Let $G$ be a Young function satisfying \eqref{G1} and let $\xi^-(t)=\min\{t^{p^-},t^{p^+}\}$, $\xi^+(t)=\max\{t^{p^-},t^{p^+}\}$, for all $t\geq0$.
    Then
    \begin{itemize}
      \item[(i)] $\xi^-(\|u\|_{L^G(\R^N)})\leq\Phi_{G,\R^N}(u)\leq\xi^+(\|u\|_{L^G(\R^N)})\ \text{for}\ u\in L^{G}(\R^N)$,
      \item[(ii)] $\xi^-([u]_{W^{s,G}(\R^N)})\leq \Phi_{s,G,\R^N}(u) \leq\xi^+([u]_{W^{s,G}(\R^N)})\ \text{for}\ u\in W^{s,G}(\R^N)$.
    \end{itemize}
 \end{lema}

\begin{lema}\label{smallemb}\cite[Lemma 3.3]{Sabri2}
  Let $\Omega$ be a bounded subset of $\mathbb{R}^{N}$ with $C^{0,1}$-regularity and bounded boundary. Let $G$ be a Young function satisfying condition \eqref{G1}. Then, given $0<s'<s<1$, the embedding
  $$
    W^{s,G}(\Omega)\hookrightarrow W^{s',1}(\Omega),
  $$
  is continuous.
\end{lema}


\section{Proof of the Strauss Radial embedding}

\begin{proof}[Proof of Theorem \ref{Strauss-embedding}]
  Let $(u_{n})\subset W^{s,G}_{rad}(\mathbb{R}^{N})$ be a bounded sequence. Since $W^{s,G}_{rad}(\mathbb{R}^{N})$ is a reflexive space, up to subsequence, still denoted by $u_n$,  $$u_{n}\rightharpoonup0\ \text{in}\ W^{s,G}_{rad}(\mathbb{R}^{N}).$$
  Using the continuous embedding $W^{s,G}(\mathbb{R}^{N})\hookrightarrow L^{G}(\mathbb{R}^{N})$, there is $C>0$, such that
  \begin{equation}\label{Mun}
    \int_{\mathbb{R}^{N}}G(u_{n})dx<C.
  \end{equation}
  As $u_{n}$ is  radially symmetric, for all $n\in\mathbb{N}$, then
  \begin{equation}\label{y1y2}
    \int_{B_{r}(y_{1})}G(u_{n})dx=\int_{B_{r}(y_{2})}G(u_{n})dx,\quad \text{for all } y_{1},y_{2}\in\mathbb{R}^{N},\ |y_{1}|=|y_{2}|.
  \end{equation}
   Let us fix $r>0$. In the sequel, for each $y\in\mathbb{R}^N$, $|y|>r$ we denote by $\gamma(y)$ the maximum of those integers $j\geq1$ such that there exist $y_1,y_2,\ldots,y_j\in\mathbb{R}^N$, with $|y_1|=|y_2|=\ldots=|y_j|=|y|$ and $$B_r(y_i)\cap B_r(y_k)=\emptyset,$$ whenever $i\neq k$. It is easy to see that
  \begin{equation}\label{NR}
    \gamma(y)\rightarrow \infty,\ \text{as}\ |y|\rightarrow\infty.
  \end{equation}
Let $y\in\mathbb{R}^N$, $|y|>r$, we choose $y_1,...,y_{\gamma(y)}\in\mathbb{R}^N$ as above,
  then by \eqref{Mun}, \eqref{y1y2} and \eqref{NR}, we obtain
  \begin{align*}
   C& \geq \int_{\mathbb{R}^N}G(u_n(x))dx\geq \sum_{i=1}^{\gamma(y)}\int_{B_r(y_i)}G(u_n(x))dx\\
   &\geq \gamma(y) \int_{B_r(y)}G(u_n(x))dx.
  \end{align*}
  Therefore $$\int_{B_r(y)}G(u_n(x))dx\leq\frac{C}{\gamma(y)}\rightarrow0,\ \ \text{as}\ \ |y|\rightarrow+\infty.$$

 So that, for arbitrary $\epsilon>0$ there exists $R_{\epsilon}>0$ such that
  \begin{equation}\label{R3}
    \sup_{|y|\geq R_{\epsilon}}\int_{B_{r}(y)}G(u_{n})dx\leq \epsilon,\ n\in \mathbb{N}.
  \end{equation}
By Proposition \ref{ceb}, the embedding $$W^{s,G}(B_{r+R_{\epsilon}}(0))\hookrightarrow L^{G}(B_{r+R_{\epsilon}}(0))$$ is compact and thus $u_{n}\rightarrow0$ in $L^{G}(B_{r+R_{\epsilon}}(0))$, which implies

    $$\int_{B_{r+R_{\epsilon}}(0)}G(u_{n})dx\rightarrow0,\ \text{as}\ n\rightarrow\infty,$$
  and so,
  \begin{equation}\label{R3r}
   \sup_{|y|< R_{\epsilon}}\int_{B_{r}(y)}G(u_{n})dx\rightarrow0,\ \text{as}\ n\rightarrow\infty.
  \end{equation}
  Combining \eqref{R3} and \eqref{R3r} and by applying Theorem \ref{Lions1}, we deduce  that $u_{n}\rightarrow0$ in $L^{\Psi}(\mathbb{R}^{N})$. The proof is completed.
\end{proof}
\section{Proofs of Lions and Lieb type results}

\begin{proof}[Proof of Theorem \ref{Lions2}]
Since
$$\lim\limits_{\substack{ t\to 0}}\frac{\Psi\left(\vert t\vert\right)}{G\left(\vert t\vert\right)}=0,$$
given $\epsilon > 0$, there exists $\tau > 0$ such that
\begin{equation}\label{p1}
\Psi(\vert t\vert)\leq \frac{\epsilon}{3M_1}G(\vert t\vert),\qquad  t\in [-\tau,\tau],
\end{equation}
where
\begin{equation*}
M_1=\sup_{n}\displaystyle\int_{\mathbb{R}^{N}} G(\vert u_n\vert) dx.
\end{equation*}
Moreover, as
$$\lim\limits_{\substack{ t\to \infty}}\frac{\Psi\left(\vert t\vert\right)}{G_{*}\left(\vert t\vert\right)}=0$$
 there exists $T> 0$ such that
\begin{equation}\label{D2}
\Psi(\vert t\vert)\leq \frac{\epsilon}{3M_2}G_{*}\left(\vert t\vert\right),\qquad \vert t\vert >T,
\end{equation}
where
\begin{equation*}
M_2=\sup_{n}\displaystyle\int_{\mathbb{R}^{N}}G_{*}(\vert u_n\vert) dx.
\end{equation*}
Therefore,
\begin{align*}
\int_{\mathbb{R}^{N}} \Psi\left(\vert u_n\vert\right)dx &\leq \left( \int_{\left[\vert u_n\vert\leq \tau\right]}+\int_{\left[\tau<\vert u_n\vert\leq T\right]} +\int_{\left[\vert u_n\vert>T\right]} \right) \Psi\left(\vert u_n\vert\right)dx\\
&\leq \frac{2\epsilon}{3}+ \Psi(T)\operatorname{mes}\left(\left[\left|u_{n}\right|>\tau\right]\right).
\end{align*}
Using the fact that $\epsilon$ is arbitrary and that $\operatorname{mes}\left(\left[\left|u_{n}\right|>\tau\right]\right)\rightarrow 0$ as $n\to+\infty$, the desired result holds.
\end{proof}
For the proof of the Theorem \ref{Lieb} (Lieb type result), we need to establish the following lemma that is a key point.
\begin{lema}\label{LemmaLieb}
Let $G \in C^{1}\left([0,+\infty)\right)$ be a Young function  satisfying \eqref{G1p}and $u \subset$ $W^{s, G}\left(\mathbb{R}^{N}\right)\setminus\{0\}$ such that
 $\Phi_{s,G,\mathbb{R}^{N}}(u)^{1/p^{\pm}} \leq M .$ Then, there is $y_0 \in \mathbb{R}^{N}$  depends on $u$ and $C_0 ,C_2> 0$ does not depend on $u$ such that
 \begin{align}
&\left[1+C_2\left( 1+M\left( \int_{\mathbb{R}^{N}} G\left(\vert u/2\vert\right) dx\right)^{-1}\right)\right]^{\frac{N}{s}}\text{mes} \left[  K(y_0)\cap \text{supp}(u)\right]\geq C_0,
\end{align}
where $K(z) =\displaystyle{\prod_{i=1}^N\left(z_i-\frac{1}{2},z_i+\frac{1}{2}\right) }$
for all $z\in \mathbb{R}^{N}$.
\end{lema}

\begin{proof}
{\bf Claim 1:}
First of all we claim that there is $y_0\in \mathbb{R}^{N}$ such that
\begin{equation}\label{c0}
 \Phi_{s,G,K(y_0)}(u)^{1/p^{\pm}}<\left( 1+M\left( \int_{\mathbb{R}^{N}} G\left(\vert u/2\vert\right) dx\right)^{-1}\right) \left( \int_{K(y_0)} G\left(\vert u/2\vert\right) dx\right).
\end{equation}
Otherwise,
\begin{equation*}
 M\geq \Phi_{s,G,\mathbb{R}^{N}}(u)^{1/p^{\pm}}\geq\left( 1+M\left( \int_{\mathbb{R}^{N}} G\left(\vert u/2\vert\right) dx\right)^{-1}\right)  \left( \int_{\mathbb{R}^{N}} G\left(\vert u/2\vert\right) dx\right) >M
\end{equation*}
the contradiction holds. Thus the Claim 1.

{\bf Claim 2:} Given $0<s^{\prime}<s<1,$ it holds that $G\left(\vert u/2\vert\right) \in W^{s^{\prime},1}\left(K(y_0)\right)$ and $\left[ G\left(\vert u/2\vert\right) \right]_{s^{\prime},1,K(y_0)}\leq C_2\Phi_{s,G,K(y_0)}(u)^{1/p^{\pm}}$ for some $C_2>0.$

Indeed, since $G$ is increasing, then $G\left(\vert u/2\vert\right)\in L^1(K(y_0)).$ On the other hand, there exists $k>0$ such that
\begin{align}\label{Lie1}
\left[ G\left(\vert u/2\vert\right) \right]_{s^{\prime},1,K(y_0)}&=
\int_{K(y_0)}\int_{K(y_0)} \frac{\left\vert G\left(\vert u/2\vert\right)(x)-G\left(\vert u/2\vert\right)(y)\right\vert}{\vert x-y\vert^{N+s^{\prime}}}\, dxdy\nonumber\\
&\leq \frac{k}{2}\int_{K(y_0)}\int_{K(y_0)}\frac{|u(x)-u(y)|}{\vert x-y\vert^{N+s^{\prime}}}\, dxdy=\frac{k}{2}\left[ u\right]_{s^{\prime},1,K(y_0)}.
\end{align}
According to Lemma \ref{smallemb}, there exists $C_1>0$ such that
\begin{equation}\label{Lie2}
  \left[ u\right]_{s^{\prime},1,K(y_0)}\leq C_1 \left[ u\right]_{W^{s,G}(K(y_0))}<\infty.
\end{equation}
The first assertion of the Claim 2 is proved.

Moreover, by \eqref{Lie1}, \eqref{Lie2} and Lemma \ref{ineq1}, we infer that
\begin{align*}
\left[ G\left(\vert u/2\vert\right) \right]_{s^{\prime},1,K(y_0)}&\leq \frac{k}{2}C_1 \left[ u\right]_{W^{s,G}(K(y_0))}\\
&\leq C_2\Phi_{s,G,K(y_0)}(u)^{1/p^{\pm}}.
\end{align*}
Thus the Claim 2 is proved.

Using the continuous  Sobolev embedding
$$
W^{s,1}\left(  K(y_0)\right)\hookrightarrow W^{s^{\prime},1}\left(  K(y_0)\right) \hookrightarrow L^{1^*}\left( K(y_0)\right)
$$
 where $1^* =\frac{N}{N-s}$, there is $C_3 > 0$ such that for all $ w\in W^{s^{\prime},1}\left(  K(y_0)\right)$,
$$
C_3\Vert w\Vert_{ L^{1^*}\left( K(y_0)\right)}\leq  \int_{K(y_0)} \vert w\vert\, dx+\left[w \right]_{s^{\prime},1,K(y_0)}.
$$
Hence for all $ u\in W^{s,G}\left( \mathbb{R}^{N} \right)$
\begin{equation} \label{c3}
 C_3\Vert G\left(\vert u/2\vert\right)\Vert_{ L^{1^*}\left( K(y_0)\right)}\leq \int_{K(y_0)} \vert  G\left(\vert u/2\vert\right)\vert dx+\left[ G\left(\vert u/2\vert\right) \right]_{s^{\prime},1,K(y_0)}.
\end{equation}
Using \eqref{c0} and the Claim 2, we obtain
\begin{align*}
  C_3\Vert G\left(\vert u/2\vert\right)\Vert_{ L^{1^*}\left( K(y_0)\right)}&\leq\int_{K(y_0)} \vert  G\left(\vert u/2\vert\right)\vert dx+C_2\Phi_{s,G,K(y_0)}(u)^{1/\alpha}\\
  &\leq\int_{K(y_0)} \vert  G\left(\vert u/2\vert\right)\vert dx\\
  &+C_2 \left( 1+M\left( \int_{\mathbb{R}^{N}} G\left(\vert u/2\vert\right) dx\right)^{-1}\right) \left( \int_{K(y_0)} G\left(\vert u/2\vert\right) dx\right)\\
  &\leq\int_{K(y_0)} \vert  G\left(\vert u/2\vert\right)\vert dx
  \left[1+C_2\left( 1+M\left( \int_{\mathbb{R}^{N}} G\left(\vert u/2\vert\right) dx\right)^{-1}\right)\right]
\end{align*}
It obvious that for all $v\in L^{1^*}(K(y_0))$,
$$\Vert v\Vert_ {L^{1}(K(y_0))}\leq \text{mes}(K(y_0))^{\frac{1^*-1}{1^*}}\Vert v\Vert_ {L^{1^*}(K(y_0))},\quad\text{with}\quad\frac{1^*-1}{1^*}=\frac{s}{N}.$$
It follows that
\begin{align*}
C_3&\leq\text{mes} \left[  K(y_0)\cap \text{supp}(u)\right]^{\frac{s}{N}}
\left[1+C_2\left( 1+M\left( \int_{\mathbb{R}^{N}} G\left(\vert u/2\vert\right) dx\right)^{-1}\right)\right]
\end{align*}
that is,
\begin{align*}
C_0 &\leq\left[1+C_2\left( 1+M\left( \int_{\mathbb{R}^{N}} G\left(\vert u/2\vert\right) dx\right)^{-1}\right)\right]^{\frac{N}{s}}\text{mes} \left[  K(y_0)\cap \text{supp}(u)\right].
\end{align*}
\end{proof}

\begin{proof}[{\bf Proof of Theorem \ref{Lieb}}]
Noting that
\begin{align*}
 \int_{\mathbb{R}^{N}} G\left( \frac{1}{2}\left(\vert u_n \vert -\frac{\tau}{2}\right)^+ \right)dx
&\geq \int_{\left[ \vert u_n \vert >\tau\right]  } G\left( \frac{1}{2}\left(\vert u_n \vert -\frac{\tau}{2}\right)^+ \right)dx\\
 &\geq G\left(\frac{\tau}{4}\right)\operatorname{mes}\left(\left[\left|u_{n}\right|>\tau\right]\right)\\
 & \geq G\left(\frac{\tau}{4}\right)\delta,
\end{align*}
that is,
\begin{equation*}
\left( \int_{\mathbb{R}^{N}} G\left( \frac{1}{2}\left(\vert u_n \vert -\frac{\tau}{2}\right)^+ \right) dx\right)^{-1}\leq \left(G\left(\frac{\epsilon}{4}\right)\delta\right)^{-1}.
\end{equation*}
 By applying Lemma \ref{LemmaLieb} to the function
$\left(\vert u_n \vert -\frac{\tau}{2}\right)^+$, we find
\begin{align*}
&\left[1+C_2\left( 1+M\left( G\left(\frac{\epsilon}{4}\right)\delta\right)^{-1}\right) \right]^{\frac{N}{s}} \text{mes} \left[  K(y_n)\cap \text{supp}\left(\left(\vert u_n \vert -\frac{\tau}{2}\right)^+\right)\right]\geq C_0.
\end{align*}
As $\text{supp}\left(\left(\vert u_n \vert -\frac{\tau}{2}\right)^+\right)=\left[\vert u_n \vert>\frac{\tau}{2}\right]$, we conclude that
\begin{equation*}
\text{mes} \left[  K(y_n)\cap \left[\vert u_n \vert>\frac{\tau}{2}\right]\right]\geq C_4.
\end{equation*}
 for some $C_4>0$ does not depend on $n$. Let $v_n (x): =u_n (x + y_n)$, we see that
\begin{equation*}
\int_{K(0)}  G\left(\vert v_n\vert\right) dx\geq \int_{ K(y_n)\cap \left[\vert u_n \vert>\frac{\tau}{2}\right]}   G\left(\vert u_n\vert\right) dx \geq G\left( \frac{\tau}{2}\right)C_4=C_5>0,\quad \text{for all } n\in \mathbb{N}.
\end{equation*}
As $(v_n)$ is bounded sequence, by the compact embedding  $W^{s,G}(\mathbb{R}^{N})\hookrightarrow L^G_{loc}(\mathbb{R}^{N})$ (see \cite[Theorem 3.1]{FBS}), there exists $v\in W^{s,G}(\mathbb{R}^{N})$ such that $v_n \rightarrow v$ in $L^G(K(0))$, up to subsequence. Thus,
\begin{equation*}
\int_{K(0)}  G\left(\vert v\vert\right) dx \geq C_5>0,
\end{equation*}
showing that $v\neq 0$, as asserted. Thus the proof.

\end{proof}

\section{Proof of Theorem \ref{m.r.1}}

\begin{defn}\label{weak.solution}
  We say that $u\in W^{s,G}(\R^N)$ is a \emph{weak solution} of \eqref{m.equation} if
$$
\langle (-\Delta_g)^s u,v \rangle +\int_{\R^N} g(|u|)\frac{u}{|u|}v\, dx= \int_{\R^N} f(u)v\,dx
$$
for all $v\in W^{s,G}(\R^N)$.
\end{defn}

As anticipated in the introduction, we will approach problem \eqref{m.equation} through the machinery of variational methods, and in particular, it will be done by using the Mountain Pass Theorem to the energy functional $\mathcal{T}$ defined in \eqref{functional}.

As it is well known, in order to follow this strategy, it is necessary to have some compactness properties on the functional, and so we shall exploit Theorem \ref{Strauss-embedding} by working with the functional $\mathrm{T}$ defined as the restriction of $\mathcal{T}$ to the
space $W^{s,G}_{rad}(\R^N),$ i.e.
$$
\mathrm{T}(u):=(\mathcal{T})|_{W^{s,G}_{rad}(\R^N)}(u),\quad u\in W^{s,G}_{rad}(\R^N).
$$
First of all, let us show that $\mathrm{T}$ satisfies the geometric Mountain Pass structure.

\begin{lema}\label{geo.condition}
$~$
The functional $\mathrm{T}$ satisfies the mountain pass geometry, that is,
\begin{itemize}
  \item [(i)] There exist $\rho>0$ and $\delta_{\rho}>0$ such that $\mathrm{T}(u)\geq \delta_{\rho}$ for any $u\in W^{s,G}_{rad}(\R^N)$ with $\|u\|_{W^{s,G}(\R^N)}=\rho.$
  \item [(ii)] There exists a strictly positive function $e\in W^{s,G}_{rad}(\R^N)$ such that $\|e\|_{W^{s,G}(\R^N)}>\rho$ and $\mathrm{T}(e)< \delta_{\rho}.$
\end{itemize}
\end{lema}

\begin{proof}
 \begin{itemize}
  \item [(i)] From \eqref{f1}-\eqref{f3}, given $\epsilon>0$, there exists $C_\epsilon>0$ such that
  \begin{equation*}
  0\leq F(t)\leq \frac{\epsilon p^+}{\theta}G(\vert t\vert )+C_\epsilon M(\vert t\vert ),\quad\text{for all } t\in \mathbb{R},
  \end{equation*}
 we get
  \begin{equation*}
  \mathrm{T}(u)\geq \iint_{\R^N\times\R^N} G(|D_s u|)\frac{dxdy}{|x-y|^N} +\left(1-\frac{\epsilon p^+}{\theta}\right)\int_{\R^N} G(|u|)dx-C_\epsilon\int_{\R^N} M(|u|)dx.
 \end{equation*}
 Hence for $\epsilon$ small enough and according to Lemma \ref{ineq1}, there exist $C_1,\ C_2>0$ such that
 \begin{equation}\label{e1}
 \mathrm{T}(u)\geq C_1\left( \xi^-([u]_{W^{s,G}(\R^N)})+\xi^-(\|u\|_{L^G(\R^N)})\right)-
 C_2 \xi^+(\|u\|_{L^M(\R^N)})
 \end{equation}
 Choosing $\rho>0$ such that
  \begin{align*}
  \|u\|_{W^{s,G}(\R^N)}=\|u\|_{L^G(\R^N)}+[u]_{W^{s,G}(\R^N)}=\rho<1
  \end{align*}
   and
 \begin{align*}
   \|u\|_{L^M(\R^N)}\leq C\left( \|u\|_{L^G(\R^N)}+[u]_{W^{s,G}(\R^N)}\right)<\rho<1 ,
  \end{align*}
 by \eqref{e1}, we obtain
 \begin{equation*}
 \mathrm{T}(u)\geq C_1\left([u]_{W^{s,G}(\R^N)}^{p^+}+\|u\|_{L^G(\R^N)}^{p^+}\right)-C_2\|u\|_{L^M(\R^N)}^{m^-}
 \end{equation*}
 witch yields
 \begin{equation*}
\mathrm{T}(u)\geq C_3 \|u\|_{W^{s,G}(\R^N)}^{p^+}-C_4\|u\|_{W^{s,G}(\R^N)}^{m^-}
 \end{equation*}
 for some positive constants $C_3$ and $C_4$. Since $0<p^+<m^-$, there exists  $\delta_{\rho}>0$ such that
 $$\mathrm{T}(u)\geq \delta_{\rho} \text{  for all } \|u\|_{W^{s,G}(\R^N)}=\rho.$$
 \item [(ii)]
By \eqref{f3}, there exist $C_5>0 $ such that
  \begin{equation}\label{pii}
 F(t)\geq C_5 \vert t\vert^\theta, \text{ for all }t\in\R.
 \end{equation}
 Let $\psi  \in C_0^\infty(\R^N)$ with $\|\psi\|_{W^{s,G}(\R^N)}>\rho$ and $\displaystyle{\int_{\R^N} \vert \psi\vert^\theta dx>0}$. Using Lemma \ref{ineq1} and \eqref{pii}, we get
\begin{align*}
\mathrm{T}(t\psi)&\leq \xi^+(t)\left( \xi^+([\psi]_{W^{s,G}(\R^N)})+\xi^+(\|\psi\|_{L^G(\R^N)})\right)-
 C_5t^\theta\int_{\R^N} \vert \psi\vert^\theta dx\\
 &\leq t^{p^+}\left( \xi^+([\psi]_{W^{s,G}(\R^N)})+\xi^+(\|\psi\|_{L^G(\R^N)})\right)-
 C_5t^\theta\int_{\R^N} \vert \psi\vert^\theta dx, \quad\text{for } t>1.
 \end{align*}
 Since $p^+<\theta$, it follows that,
$$\mathrm{T}(t\psi)\rightarrow -\infty\text{ as } t\rightarrow+\infty$$
Hence, there exists $t_0 > 0$ such that $\mathrm{T}(e)<  \delta_{\rho}$, where $e = t_0\psi$. Thus the proof
of (ii) is complete.
\end{itemize}

\end{proof}

\begin{lema}\label{cv of PS}
  The functional $\mathrm{T}$ satisfies the Palais–Smale condition at any level $c\in\R$ $((PS)_c$ condition, for short$)$,  that is for any Palais–Smale sequence at level $c$ for  $\mathrm{T}$, i.e. a sequence $(u_n)_{n\in\mathbb{N}}\subset W^{s,G}_{rad}(\R^N)$  satisfying
\begin{equation}\label{cv T}
  \mathrm{T}(u_n)\rightarrow c
\end{equation}
and
\begin{equation}\label{cv T'}
  \sup\{ |\langle\mathrm{T}^{'}u_n,v\rangle|:\ v\in W^{s,G}_{rad}(\R^N),\ \|v\|_{W^{s,G}(\R^N)}\leq 1\} \rightarrow0
\end{equation}
as $n\rightarrow+\infty$ admits a subsequence which is strongly convergent in $W^{s,G}_{rad}(\R^N)$.
\end{lema}

\begin{proof}
Let $(u_n)_{n\in\mathbb{N}}$ be a Palais–Smale sequence  at level $c$ for $\mathrm{T}$.

{\bf Claim:} $(u_n)_{n\in\mathbb{N}}$ is bounded in $W^{s,G}_{rad}(\R^N).$

Indeed, there exists $C>0$ such that
$$
C(1+\|u_n\|_{W^{s,G}(\R^N)})\geq \mathrm{T}(u_n)-\frac{1}{\theta}\langle\mathrm{T}^{'}u_n,u_n\rangle,\quad\forall n\in\mathbb{N}.
$$
According to \eqref{f3},
\begin{align*}
   C(1+\|u_n\|_{W^{s,G}(\R^N)})&\geq \left(\frac{\theta-p^+}{\theta}\right)\left[\iint_{\R^N\times\R^N} G(|D_s u_n|)\frac{dxdy}{|x-y|^N} +\int_{\R^N} G(|u_n|)dx\right]\\
   &\geq \left(\frac{\theta-p^+}{\theta}\right)\left[\xi^-([u_n]_{W^{s,G}(\R^N)})+\xi^-(\|u_n\|_{L^G(\R^N)})\right].
\end{align*}

Arguing by contradiction, suppose that $\|u_n\|_{W^{s,G}(\R^N)}\to\infty$, up to a subsequence. This yields that the following $(a)$, $(b)$ or $(c)$ occurs:
\begin{itemize}
  \item [(a)] $[u_n]_{W^{s,G}(\R^N)}\to \infty$ and $\|u_n\|_{L^G(\R^N)}\to\infty.$
  \item [(b)] $[u_n]_{W^{s,G}(\R^N)}\to \infty$ and $\|u_n\|_{L^G(\R^N)}$ is bounded.
  \item [(c)] $[u_n]_{W^{s,G}(\R^N)}$ is bounded and $\|u_n\|_{L^G(\R^N)}\to\infty.$
\end{itemize}

In  case $(a)$, according to Lemma \ref{ineq1}, there exists $C_1>0$ such that
$$
C(1+\|u_n\|_{W^{s,G}(\R^N)})\geq\left(\frac{\theta-p^+}{\theta}\right)\left([u_n]_{W^{s,G}(\R^N)}^{p^-}+\|u_n\|_{L^G(\R^N)}^{p^-}\right)\geq C_1 \|u_n\|_{W^{s,G}(\R^N)}^{p^-},
$$
for $n$ large enough, which is a contradiction.

In case $(b)$, also according to Lemma \ref{ineq1}, for $n$ large enough, there exists $C_2>0$ such that
$$
C_3(1+[u_n]_{W^{s,G}(\R^N)}) \geq C(1+\|u_n\|_{W^{s,G}(\R^N)})\geq C_2[u_n]_{W^{s,G}(\R^N)}^{p^-},
$$
which is a contradiction. The case $(c)$ is similar to $(b)$, we conclude that $(u_n)_{n\in\mathbb{N}}$ is bounded in $W^{s,G}_{rad}(\R^N)$.

Since $W^{s,G}_{rad}(\R^N)$ is a reflexive space, up to a subsequence, still denoted by $(u_n)_{n\in\mathbb{N}}$, there exists $u\in W^{s,G}_{rad}(\R^N)$ such that
\begin{equation}\label{cv.wekly}
  u_n\rightharpoonup u\quad \text{weakly in}\quad W^{s,G}_{rad}(\R^N)\quad \text{as}\ n\to\infty.
\end{equation}

By the compact embedding \eqref{Strauss-embedding}, we infer that
\begin{equation}\label{cv.wekly.psi}
  u_n\to u\quad \text{in}\quad L^{M}(\R^N)\ \text{and}\ L^{G}(\R^N)\quad \text{as}\ n\to\infty.
\end{equation}
On one hand, by \eqref{cv T'}, we find that
\begin{align*}
  \lim_{n\to\infty} \langle\mathrm{T}^{'}u_n,u_n-u\rangle=0.
\end{align*}
On the other hand, by \eqref{f1}- \eqref{f2},  \eqref{cv.wekly.psi}, Young inequality, boundedness of $(u_n)_n$ and \cite[Lemma 2.9]{FBS}, we deduce that
\begin{align*}
   \int_{\R^N}f(u_n)(u_n-u)\,dx\to0,\quad\text{as}\ n\to\infty.
\end{align*}
Since
$$
\langle\mathrm{T}^{'}u_n,u_n-u\rangle=\langle\mathcal{J}^{'}u_n,u_n-u\rangle
+\langle\mathcal{I}^{'}u_n,u_n-u\rangle-\int_{\R^N}f(u_n)(u_n-u)\,dx,
$$
then
$$
\langle\mathcal{J}^{'}u_n+\mathcal{I}^{'}u_n,u_n-u\rangle
\to0,\quad\text{as}\quad n\to\infty.
$$
According to \cite[Lemma 4.9]{Sabri2}, $(\mathcal{J}+\mathcal{I})^{'}$ is of type $(S_+)$ and so $(u_n)$ converges strongly to $u$ in $W^{s,G}(\R^N).$
\end{proof}

\begin{proof}[{\bf Conclusion of the proof of Theorem \ref{m.r.1}}]
  Thanks to Lemma \ref{geo.condition} and Lemma \ref{cv of PS}, by applying the Mountain Pass Theorem,
  there exists $u_1 \in W^{s,G}_{rad}(\R^N)$ a critical point of $\mathrm{T}$, that is
  $$
  \langle \mathrm{T}^{'}(u_1),\varphi\rangle=0\quad\text{for all}\quad \varphi \in W^{s,G}_{rad}(\R^N).
  $$

  It remains to prove that $W^{s,G}_{rad}(\R^N)$ is a natural constraint for $\mathcal{T},$ i.e. $u_1$ is a  critical point of $\mathcal{T}$, that is in sense of definition \ref{weak.solution}. To this aim we will use the well known
principle of symmetric criticality  \cite{Palai} that we recall briefly.
{\it

\hspace{0.3 cm}  Let $(E,\|.\|_E)$ be a reflexive Banach space. Suppose that $\textbf{G}$ is a subgroup of isometries $g:E\rightarrow E$.
Consider the $\textbf{G}$-invariant closed subspace of $E$ $$\Sigma=\{u\in E:\ gu=u\ \text{for all}\ g\in \textbf{G}\}.$$

\begin{lema}\cite[Proposition 3.1]{daniel}\label{PCS}
  Let $E$,$\textbf{G}$ and $\Sigma$ be as before and let $J$ be a $C^1$ functional defined on $E$ such that $J\circ g=J$ for all $g\in \textbf{G}$. Then $u\in \Sigma$ is a critical point of $J$ if and only if $u$ is a critical point of $J|_{\Sigma}$.
\end{lema}

}
Finally, we have all the ingredients to complete the proof of Theorem \ref{m.r.1}.

Let $SO(N)$ denote the special orthogonal group, that is
$$
SO(N)=\{A\in \mathcal{M}_{N\times N}(\mathbb{R}):\ \ A^{t}A=I_N\ \text{and}\ det(A)=1\}.
$$
Consider the following subgroup of linear operators of $W^{s,G}(\R^N)$ in itself
$$
\textbf{G}=\bigg{\{}a\colon W^{s,G}(\R^N)\rightarrow W^{s,G}(\R^N):\quad \ au=u\circ A,\ \text{where}\ A\in SO(N)\bigg{\}}.
$$
We have
$$
W^{s,G}_{rad}(\R^N)=\{u\in W^{s,G}(\R^N):\quad gu=u\ \text{for all}\ g\in \textbf{G}\}.
$$

Let's prove that $\textbf{G}$ is a subgroup of isometries of $W^{s,G}(\R^N)$:

\hspace{0.3 cm} Fix $u$ in $W^{s,G}(\R^N)$, let  $A\in SO(N)$ and $x,y\in\R^N$, we have $|x-y|=|A(x-y)|=|Ax-Ay|=|x'-y'|$. Let $a\in \textbf{G}$, we have
\begin{align*}
  \|au\|_{W^{s,G}(\R^N)}&=
   \inf\left\{\lambda>0\colon \int_{\R^N} G\left(\frac{|u(Ax)|}{\lambda}\right)\,dx \leq 1\right\}\\
   &+\inf\left\{\lambda>0\colon \iint_{\R^N\times\R^N} G\left(\frac{ |u(Ax)-u(Ay)|}{\lambda|x-y|^s}\right) \frac{dxdy}{|x-y|^N} \leq 1\right\}\\
  &=\inf\left\{\lambda>0\colon \int_{\R^N} G\left(\frac{|u(x')|}{\lambda}\right)\,dx' \leq 1\right\}\\
   &+\inf\left\{\lambda>0\colon \iint_{\R^N\times\R^N} G\left(\frac{ |u(x')-u(y')|}{\lambda|x'-y'|^s}\right) \frac{dx'dy'}{|x'-y'|^N} \leq 1\right\} \\
&=  \|u\|_{W^{s,G}(\R^N)}.
\end{align*}

In order to apply Lemma \ref{PCS} to the functional $\mathcal{T}$ , we need to show that $\mathcal{T} \circ a = \mathcal{T}$ for all $a \in \textbf{G}$.
Fixed $u\in W^{s,G}(\R^N)$, for all $a \in \textbf{G}$ it comes that
\begin{align*}
  (\mathcal{T}\circ a)(u)&=\iint_{\R^N\times\R^N} G\left(\frac{ |u(Ax)-u(Ay)|}{|x-y|^s}\right) \frac{dxdy}{|x-y|^N}+\int_{\R^N} G\left(|u(Ax)|\right)\,dx-\int_{\R^N} F\left(|u(Ax)|\right)\,dx\\
  &=\iint_{\R^N\times\R^N} G\left(\frac{ |u(x')-u(y')|}{|x'-y'|^s}\right) \frac{dx'dy'}{|x'-y'|^N}+\int_{\R^N} G\left(|u(x')|\right)\,dx'-\int_{\R^N} F\left(|u(x')|\right)\,dx'\\&=\mathcal{T}(u).
\end{align*}

Then, Lemma \ref{PCS} implies that $u$ is a critical point of $\mathcal{T}$ in the whole space $W^{s,G}(\R^N)$. Thus, $u$ is a weak solution
of \eqref{m.equation} in the sense of definition \ref{weak.solution}.
\end{proof}
\section{Proof of Theorem \ref{m.r.2}}
It is well known that if $p^->1$, then $G$ and $\widetilde G$ satisfies the  $\Delta_2$-condition and therefore  the functional $\Phi:=\mathcal{J}+\mathcal{I}\in C^{1}(W^{s,G}(\R^N),\mathbb{R})$( see). But,  if $p^-=1$,  $\tilde{G}$ does not satisfy the $\Delta_2$-condition, and as a consequence, $ W^{s,G}(\R^N)$ is non reflexive anymore.
An immediate consequence, we cannot guarantee that $\mathcal{T}$ belongs to $C^{1}(W^{s,G}(\R^N),\mathbb{R})$. However, the functional $\mathcal{F}: W^{s,G}(\R^N) \to \mathbb{R} $ given by \eqref{func.1}
belongs to $C^{1}( W^{s,G}(\R^N),\mathbb{R}),$(See appendix for details).
Furthermore, we know that $\Phi$ is strictly convex   and  l.s.c. with respect to the weak$^*$ topology.\\
Concerning the lack of regularity of $\Phi$ and according the above commentaries, in the present part of the paper we will use a minimax method developed by Szulkin  \cite{Szulkin}. In this sense, we will say that $u \in W^{s,G}(\R^N)$ is a critical point for $\mathcal{T}$ if $0 \in \partial \mathcal{T}(u)$, where the  sub-differential of the functional $\mathcal{T}$ at  a point $u\in W^{s,G}(\R^N)$ is the following
\begin{align*}
\partial \mathcal{T}(u)=\left\lbrace w\in W^{-s,\tilde{G}}(\R^N):\,\Phi(v)-\Phi(u)-\langle \mathcal{F}'(u), v-u\rangle\geq\langle w,v-u\rangle,\right. \\
 \left. \text{ for all }v \in W^{s,G}(\R^N)  \right\rbrace .
\end{align*}

 Then $u \in W^{s,G}(\R^N)$ is a critical point of $\mathcal{T}$ if, and only if,
\begin{equation} \label{E1}
\Phi(v)-\Phi(u)\geq \int_{\Omega}f(u)(v-u)\,dx, \quad \text{ for all } v \in   W^{s,G}(\R^N).
\end{equation}

\begin{lema}\label{weak.solution.w}
 If $u \in  W^{s,G}(\R^N)$ be a critical point of $\mathcal{T}$ in $ W^{s,G}(\R^N)$ in \eqref{E1} sense, then $u$ is a weak solution for \eqref{m.equation}, that is, \begin{equation}\label{equa.weak}
\langle (-\Delta_g)^s u,v \rangle +\int_{\R^N} g(|u|)\frac{u}{|u|}v\, dx=\int_{\Omega}f(u)v\,dx, \quad \forall v \in  W^{s,G}(\R^N).
\end{equation}
\end{lema}

\begin{proof}
 Since $G$ satisfies the $\Delta_2$-condition, we claim that $\Phi $ is
G\^ateaux differentiable, that is,
$\langle\Phi'(u),v\rangle$ exists for all $u,\ v \in  W^{s,G}(\R^N)$ with
\begin{equation}\label{differential}
\langle\Phi'(u),v\rangle=\langle (-\Delta_g)^s u,v \rangle +\int_{\R^N} g(|u|)\frac{u}{|u|}v\, dx
\end{equation}
Indeed, for each $v \in  W^{s,G}(\R^N)$ and $t \in [-1, 1]\setminus \{0\}$,
\begin{equation*}
G(|D_s u + t D_s v|) - G(|D_s u|) = t g(|D_s u +\theta  t D_s v|)\frac{(D_s u +\theta  t D_s v)}{|D_s u +\theta  tD_s v|} D_s v,
\end{equation*}
for some $\theta \in (0, 1)$.Consequently,
\begin{equation*}
\left\vert \frac{G(|D_s u + t D_s v|) - G(|D_s u|)}{t}\right\vert = g(|D_s u +\theta  t D_s v|)|D_s v|.
\end{equation*}
By \eqref{ineb} and Young inequality \eqref{Young}, there is $C > 0$ such
that
\begin{equation*}
g(|D_s u +\theta  t D_s v|)|D_s v|\leq C G(|D_s u| + |D_s v|)+G(|D_s v|)\in L^1(\R^N\times \R^N),
\end{equation*}
and
\begin{equation*}
g(| u +\theta  t  v|)|v|\leq C G(|u| + |v|)+G(|v|)\in L^1(\R^N).
\end{equation*}
Now, by applying Lebesgue dominated convergence theorem, we derive that
\begin{equation*}
\lim_{t\to 0} \frac{\Phi( u+tv)-\Phi( u)}{t}=\iint_{\R^N\times\R^N} g(|D_s u|) \frac{D_s u}{|D_s u|}  D_s v \,d\mu+\int_{\R^N} g(|u|)\frac{u}{|u|}v\, dx,
\end{equation*}
showing \eqref{differential}.\\
Recalling that the functional $\mathcal{F}$ given by  \eqref{func.1} belongs to $ C^{1}( W^{s,G}(\R^N),\mathbb{R})$ with
$$\mathcal{F}'(u)v=\int_{\R^N} f(u)v\,  dx,\qquad \text{ for all } u,\ v \in  W^{s,G}(\R^N),$$
it follows that $\mathcal{T}$ is G\^ateaux differentiable with
\begin{equation}
\langle \mathcal{T}'(u),v\rangle=\langle \Phi'(u),v\rangle
-\mathcal{F}'(u)v,\qquad \text{ for all } u,\ v \in  W^{s,G}(\R^N),
\end{equation}
i.e,
\begin{equation*}
\langle \mathcal{T}'(u),v\rangle=\langle (-\Delta_g)^s u,v \rangle +\int_{\R^N} g(|u|)\frac{u}{|u|}v\, dx-\int_{\R^N}f(u)v\,dx,  \text{ for all } u,\ v \in  W^{s,G}(\R^N).
\end{equation*}
Since $u$ is a critical point of $\mathcal{T}$, then
\begin{equation*}
\Phi(w)-\Phi(u)\geq \int_{\R^N}f(u)(w-u)\,dx, \quad\text{for all } w \in   W^{s,G}(\R^N).
\end{equation*}
Hence, for each $ v \in  W^{s,G}(\R^N)$ and $t > 0$,
\begin{equation*}
\frac{\Phi( u+tv)-\Phi( u)}{t} \geq  \int_{\R^N}f(u)v\,dx.
\end{equation*}
Letting $t \to 0$, we obtain
\begin{equation*}
\langle (-\Delta_g)^s u,v \rangle +\int_{\R^N} g(|u|)\frac{u}{|u|}v\, dx-\int_{\R^N}f(u)v\,dx \geq 0, \quad\text{for all } v \in  W^{s,G}(\R^N).
\end{equation*}
The last inequality ensures that
\begin{equation*}
\langle (-\Delta_g)^s u,v \rangle +\int_{\R^N} g(|u|)\frac{u}{|u|}v\, dx-\int_{\R^N}f(u)v\,dx =0, \quad\text{for all } v \in  W^{s,G}(\R^N).
\end{equation*}
Thus the proof.
\end{proof}

\subsection{Proof of Theorem \ref{m.r.2}}

\begin{lema}
Suppose that the assumption \eqref{G1p} is satisfied, then the functional $\mathcal{T}$ satisfies the mountain pass geometry of Theorem \ref{MPT.cerami}.
\end{lema}

\begin{proof}
 The proof is similar to the Lemma \ref{geo.condition}.
\end{proof}

\begin{proof}[{\bf Conclusion of the proof of Theorem \ref{m.r.2}}]
The previous Lemma allows us to  apply Corollary \ref{cor.MPT} which guaranties the existence of a $(PS)_{c}$ sequence $(u_n)\subset W^{s,G}(\R^N)$  associated with the mountain pass level $c$ of $\mathcal{T},$ that is, $\mathcal{T}(u_n)\to c>0$ and $\tau_n\to 0$ in $\R$ such that
\begin{equation} \label{sequencia2}
\Phi(v)-\Phi(u_n) \geq \int_{\R^N}f(u_n)(v-u_n)\,dx- \tau_n\|v-u_n\|,\quad \text{for all }v \in W^{s,G}(\R^N).
\end{equation}

By the similar ideas as the claim in the proof of Lemma \ref{cv of PS}, the sequence $(u_n)\subset W^{s,G}(\R^N)$ is bounded. In the light of \cite[Theorem 3.1]{FBS}, there exist $u\subset W^{s,G}(\R^N)$ and a subsequence of $(u_n)$, still denoted $(u_n)$, such that $u_n\rightarrow u $ in $L^G_{loc}(\R^N)$.\\
 On one hand, we have
$$
\mathcal{F}'(u_n)(v-u_n)  \rightarrow \mathcal{F}'(u)(v-u),\quad \text{for all } v \in W^{s,G}(\R^N),
$$
equivalent to
\begin{equation}\label{cvf1}
 \int_{\R^N}f(u_n)(v-u_n)\,dx \rightarrow \int_{\R^N}f(u)(v-u)\,dx,\quad \text{for all } v \in W^{s,G}(\R^N).
\end{equation}
On the other hand,
\begin{equation}\label{cvphi}
  \liminf_{n} \Phi(u_n)\geq \Phi(u).
\end{equation}

Combining \eqref{cvf1} and \eqref{cvphi} and passing to the limit in \eqref{sequencia2}, we obtain
$$
\Phi(v)-\Phi(u)\geq \int_{\R^N}f(u)(v-u)\,dx, \quad \text{for all }v \in   W^{s,G}(\R^N),
$$
which mains that $u$ is a critical point of $\mathcal{T}$ in \eqref{E1} sense. In light of Lemma \ref{weak.solution.w}, $u$ is a weak solution of \eqref{m.equation}.\\

 It remains to show that $u$ is non trivial weak solution. To do so we argue by contradiction. Let's suppose that $u=0$, and in this case, the
next lemma is crucial in our approach
\begin{lema}\label{plem}
Assume \eqref{G1p} and let $f : \R \rightarrow \R$ be a continuous function satisfying \eqref{f1}-\eqref{f2}.
Let $(u_n) \subset W^{s,G}(\mathbb{R}^{N})$ a sequence satisfying the condition \eqref{lim}. Then
 $$\int_{\R^N}f(u_n)(u_n)\,dx \rightarrow 0 \quad\text{as}\quad n\to\infty.$$

\end{lema}
\begin{proof}
From \eqref{f1} and \eqref{G1p}, given $\epsilon> 0$, there is $\delta > 0$ such that
$$ tf(t)\leq \epsilon p^+ G(\vert t\vert ),\quad \text{for all}\quad\vert t\vert\leq \delta.$$
From \eqref{f2} and \eqref{m1}, there exists $C_\epsilon>0$ such that
$$ tf(t)\leq C_\epsilon m^+ M(\vert t\vert),\quad \text{for all}\quad\vert t\vert\geq \delta.$$
Hence,
$$\int_{\R^N}f(u_n)(u_n)\,dx\leq \epsilon p^+ \int_{\R^N}G(\vert u_n\vert )\,dx+C_\epsilon m^+\int_{\R^N}M(u_n)\,dx,$$
by Theorem \ref{Lions2}, we have the convergence
 $$
 \int_{\R^N}M(|u_n|)\,dx\to 0\quad\text{as}\quad n\to\infty.
 $$
 Fixing $D=\displaystyle{ \sup_{n\in\mathbb{N}} \left( \int_{\R^N}G(\vert u_n\vert )\,dx\right)}$, there is $n_0\in \mathbb{N}$ such that
$$\int_{\R^N}f(u_n)(u_n)\,dx\leq\epsilon  p^+ D,\quad \text{for all }n\geq n_0,$$
and the proof is completed.

\end{proof}
 {\bf Claim:} The sequence $(u_n)$ does not satisfy the condition \eqref{lim} in Theorem \ref{Lions2}.

 Otherwise, by Lemma \ref{plem}, we have the convergence
 $$\int_{\R^N}f(u_n)(u_n)\,dx \rightarrow 0 \quad\text{as}\quad n\to\infty.$$
 Since $$\langle \mathcal{T}'(u_n),u_n\rangle
 =\iint_{\R^N\times\R^N} g(|D_s u_n|) |D_s u_n|  \,d\mu +\int_{\R^N} g(|u_n|)|u_n|\, dx-\int_{\R^N}f(u_n)u_n\,dx=\circ_n(1),$$
it implies that
$$\iint_{\R^N\times\R^N} g(|D_s u_n|) |D_s u_n|  \,d\mu +\int_{\R^N} g(|u_n|)|u_n|\, dx \rightarrow 0 \quad\text{as}\quad n\to\infty$$
Combining with \eqref{G1p}, it comes that
$$\iint_{\R^N\times\R^N} G(|D_s u_n|) \,d\mu +\int_{\R^N} G(|u_n|)\, dx \rightarrow 0 \quad\text{as}\quad n\to\infty.$$
The $\Delta_2$-condition of $G$ ensure that
$$ u_n \rightarrow 0 \text{ in  }  W^{s,G}(\R^N)$$
which is a contradiction with $\mathcal{T}(u_n)\to c>0$. Thus the claim.\\

Therefore, there are $\tau,\delta>0$ such that
 $$
\operatorname{mes}\left( \left[\left|u_{n}\right|>\tau\right]\right) \geq \delta, \quad \text{for all }n \in \mathbb{N}.
$$
 By Theorem \ref{Lieb}, there is $\left(y_{n}\right) \subset \mathbb{Z}^{N}$ such that $w_{n}(x)=u_{n}\left(x+y_{n}\right)$ has a subsequence whose its limit  $w\in L_{l o c}^{G}\left(\mathbb{R}^{N}\right)\setminus \lbrace 0\rbrace$. \\
We conclude that $u$ is a nontrivial weak solution to \eqref{m.equation}.
\end{proof}

\appendix

\section{An abstract existence result}

\begin{thm}(Mountain Pass Theorem without (PS) condition)\cite[Theorem $3.1$]{Alves.Daniel}\label{MPT.cerami}
  Let $E$ be a real Banach space and $J: E \rightarrow(-\infty,+\infty]$ be a functional such that:
  \begin{enumerate}
    \item [(i)] $J(u)=\Psi(u)+\Phi(u), u \in E$, with $\Psi \in C^{1}(E, \mathbb{R})$ and $\Phi: E \rightarrow(-\infty,+\infty]$ is convex, $\Psi \not \equiv+\infty$ and is lower semicontinuous (l.s.c);
    \item [(ii)] $J(0)=0$ and $\left.J\right|_{\partial B_{\rho}} \geq \alpha$, for some constants $\rho, \alpha>0 ;$
    \item [(iii)] $J(e) \leq 0$, for some e $\notin \overline{B_{\rho}}(0)$.
  \end{enumerate}
If
$$
c:=\inf _{\gamma \in \Gamma} \sup _{t \in[0,1]} J(\gamma(t)), \quad \Gamma=\{\gamma \in C([0,1], E) ; \gamma(0)=0, J(\gamma(1))<0\}
$$
then, for a given $\epsilon>0$ there is $u_{\epsilon} \in E$ such that
$$
\left\langle\Psi^{\prime}\left(u_{\epsilon}\right), v-u_{\epsilon}\right\rangle+\Phi(v)-\Phi\left(u_{\epsilon}\right) \geq-3 \epsilon\left\|v-u_{\epsilon}\right\|,\quad \forall v \in E
$$
 and
$$
J\left(u_{\epsilon}\right) \in[c-\epsilon, c+\epsilon].
$$
\end{thm}

\begin{cor}\label{cor.MPT}
  Under the conditions of Theorem \ref{MPT.cerami}, there is a $(PS)_{c}$ sequence $\left(u_{n}\right) \subset E$ for $J$, that is, $J\left(u_{n}\right) \rightarrow c$ and
$$
\left\langle\Psi^{\prime}\left(u_{n}\right), v-u_{n}\right\rangle+\Phi(v)-\Phi\left(u_{n}\right) \geq-\tau_{n}\left\|v-u_{n}\right\|, \quad \forall v \in E
$$
with $\tau_{n} \rightarrow 0^{+}$.
\end{cor}


\end{document}